\documentclass[12pt]{article}

\usepackage{mathrsfs}

\usepackage{latexsym}
\usepackage{amsmath}
\usepackage{graphics}
\usepackage{epsfig}
\usepackage{color}
\usepackage[numbers,sort&compress]{natbib}
\usepackage{mathrsfs}
\usepackage{amssymb,amsfonts,amsmath, psfrag}
\usepackage{graphicx}
\usepackage{enumerate}
\usepackage{lineno}
\usepackage{subfigure}
\usepackage{caption}
\usepackage{overpic}
\usepackage{epstopdf}

\allowdisplaybreaks[4]
\newtheorem{theo}{Theorem}

\newtheorem{Lemma}{Lemma}

\allowdisplaybreaks [4]
\newenvironment {proof} {\noindent{\em Proof.}}{\hspace*{\fill}$\Box$\par\vspace{4mm}}

\begin{document}

\title{Uniform hypergraphs with the first two smallest spectral radii}

\author{Jianbin Zhang$^{a}$\footnote{Corresponding author. E-mail:zhangjb@scnu.edu.cn},
 Jianping Li$^{b}$\footnote{E-mail:lijp06@gdut.edu.cn}, Haiyan Guo$^{a}$\footnote{E-mail:ghaiyan@163.com}
\\[2mm]
\small $^{a}$School of Mathematical Sciences, South China Normal University, \\
\small Guangzhou 510631, P. R. China\\
\small $^{b}$Faculty of Applied Mathematics, Guangdong University of Technology, \\
\small Guangzhou 510090, P. R. China
}

\date{ }
\maketitle

\begin{abstract}
The spectral radius of a uniform hypergraph $G$ is
the  the maximum modulus of the eigenvalues of the adjacency tensor of $G$.
For $k\ge 2$, among connected $k$-uniform hypergraphs with $m\ge 1$ edges, we show that the $k$-uniform loose path with $m$ edges is the unique one with minimum spectral radius, and  we also determine the unique ones with second minimum spectral radius when $m\ge 2$.
\end{abstract}


\section{Introduction}

Let $G$ be a hypergraph with vertex set $V(G)$ and edge set $E(G)$,
where $E(G)$ is a set whose elements are subsets of $V(G)$. For an integer $k\ge 2$, if each
edge of $G$ contains exactly $k$ distinct vertices, then $G$ is
a $k$-uniform hypergraph. Two vertices $u$ and $v$ are
adjacent if $u$ and $v$ are contained in some edge. An $e$ is incident
with vertex $v$ if $v\in e$. An alternating sequence of vertices
and edges is called a path if all vertices and edges are distinct,
and  a cycle if the first and last vertices are the same, the other
vertices and all edges are distinct. If there exists a path between
any two vertices of $G$, then $G$ is  connected. A hypertree
is a connected acyclic  hypergraph. A vertex of degree one is called
 a pendant vertex. 

For a $k$-uniform hypergraph $G$ with  vertex set $V(G)=\{1, \dots, n\}$, its adjacency
tensor is the tensor
$\mathcal {A}(G)=(a_{i_1\ldots i_k})$ of  order $k$ and dimension $n$
 with $a_{i_1\ldots i_k}=\frac{1}{(k-1)!}$ if
$\{i_1\ldots i_k\}\in E(G)$, $0$ otherwise, where $i_j\in \{1, \dots, n\}$ and $j\in \{1, \dots, k\}$.

For some complex $\lambda$,
if there exists a nonzero vector $x=(x_1,\ldots,x_n)^T$ such that
$\mathcal {A}(G)x=\lambda x^{k-1}$, then $\lambda$ is called an eigenvalue of $G$ and $x$ is called the eigenvector of $G$ corresponding to $\lambda$, where
 $\mathcal {A}(G)x$ is a
$n$-dimensional vector whose $i$-th component is
\begin{equation*}\label{E1}
(\mathcal {A}(G)x)_i=\sum\limits_{i_2,\ldots, i_k=1}^{n}a_{ii_2\ldots
i_k}x_{i_2}\cdots x_{i_k}
\end{equation*}
and $x^{k-1}=(x_1^{k-1},\ldots,x_n^{k-1})^T$,
 Moreover, if $\lambda$ and $x$ are both real, then we call
$\lambda$ an $H$-eigenvalue of $G$. Note that $x^T(\mathcal {A}(G)x)=k\sum_{e\in E(G)}\prod_{v\in e}x_v$.

The spectral radius of a $k$-uniform hypergraph $G$, denoted by $\rho(G)$, is defined as the maximum modulus of eigenvalues of $\mathcal {A}(G)$.
Since $\mathcal{A}(G)$ is symmetric and thus
$\rho(G)$ is the largest $H$-eigenvalue of $\mathcal{A}(G)$, see \cite{QI}.
It is proved in \cite{FGH,PZ} that  for a connected $k$-uniform  hypergraph $G$,  $\mathcal {A}(G)$ has a unique positive eigenvector
$x$ with $\sum_{v\in V(G)}x_v^k=1$ corresponding to $\rho(G)$, which is called the principal eigenvector of $G$.

The problem  to determine the hypergraphs in some given classes of hypergraphs with maximum  spectral radius received much attention.
Li et al.~\cite{LSQ} determined the  uniform hypertree  with maximum spectral radius.
Yuan et al.~\cite{YSS} extended to this to determine the first eight uniform hypertrees  with  maximum spectral radius. Fan et al.~\cite{FTPL}
determined the hypergraphs with maximum spectral radius among  uniform hypergraphs with few edges.
Xiao et al.~\cite{XWL} determined the hypertree with  maximum spectral radius among uniform hypertrees with a given degree sequence. See \cite{} for more work on hypergraphs with maximum spectral radius in someclasses of uniform hypergraphs.
However, there is much less work on the problem  to determine the hypergraphs in some given classes of hypergraphs with minimum  spectral radius.
Li et al.~\cite{LSQ} determined the unique one with minimum spectral radius among $k$-uniform power hypertrees with fixed number of edges. Here a $k$-uniform power hypertree is a $k$-uniform hypertree in which every edge contains at least $k-2$ pendant vertices. Lu and Man \cite{LM} classified all connected $k$-uniform hypergraphs with spectral radius at most $\sqrt[k]{4}$. There seems no more result in this line.
In this note, for $k\ge 2$, we show that the $k$-uniform loose path with $m$ edges is the unique one with minimum spectral radius among connected $k$-uniform hypergraphs with $m\ge 1$ edges, and  we also determine the unique ones with second minimum spectral radius among connected $k$-uniform hypergraphs with $m\ge 2$ edges. To obtain our main result, we use the nice result (Lemma~\ref{Lem5}) from \cite{LM} and propose a hypergraph transformation that decreases the spectral radius (in Lemma~\ref{Th1}).

\section{Preminaries}

Let $r\ge 1$, $G$ be a hypergraph with $u\in V(G)$ and $e_1,\ldots,e_r\in E(G)$. Suppose that $v_i\in e_i$ and $u\not\in e_i$ for $i=1,2,\ldots,r$. Let
$e_i'=(e_i\backslash \{v_i\})\cup \{u\}$ for $i=1,\ldots,r$. Suppose that   $e_i'\not \in E(G)$ for $i=1,\ldots,r$.
Let  $G'$ be the hypergraph
obtained from $G$ by deleting $e_1,e_2,\ldots,e_r$ and adding
$e_1',e_2',\ldots,e_r'$.   Then we say that $G'$ is obtained from
$G$ by moving $(e_1,e_2,\ldots,e_r)$ from $(v_1,v_2,\ldots,v_r)$ to $u$.

\begin{Lemma}\label{Lem1} \cite{LSQ} Let $r\ge 1$, $G$ be a hypergraph with $u\in V(G)$ and $e_1,\ldots,e_r\in E(G)$.
If $G'$ is obtained from
$G$ by moving $(e_1,e_2,\ldots,e_r)$ from $(v_1,v_2,\ldots,v_r)$ to $u$,
and  $x_u\ge \max\limits_{1\le i\le r}x_{v_i}$, then
$\rho(G')>\rho(G)$.
\end{Lemma}


\begin{Lemma}\label{Lem2} \cite{XWL} Let $G$ be a connected  $k$-uniform
hypergraph, and  $e=U_1\cup U_2$, $f=V_1\cup V_2$ be two edges of $G$, where $e\cap f=\emptyset$, and $1\le |U_1|=|V_1|<k$.
Let $e'=U_1\cup V_2$ and $f'=V_1\cup U_2$. Suppose that $e',f'\not\in E(G)$.
Let $G'$ be the hypergraph obtained from $G$ by deleting edges $e$ and $f$ and adding edges $e'$ and $f'$.
 Let $x$ be the principle eigenvector of $G$. If $x_{U_1}\geq
x_{V_1}$, $x_{U_2}\leq x_{V_2}$, and one is strict, then
$\rho(G)<\rho(G')$.
\end{Lemma}


A  path $(u_0,e_1,u_1,\ldots,e_p,u_p)$ in a $k$-unoform hypergraph $G$
 is called a pendant path at $u_0$ if $d(u_0)\ge 2, d(u_i)=2$ for $i=1,2,\ldots,p-1$, $d(u_p)=1$ and $d(u)=1$ for any $u\in e_i\setminus \{u_{i-1},u_i\}$ with $i=1,2,\ldots,p$. If $p=1$, then it is a pendant edge at $u_0$.

For a $k$-uniform hypergraph $G$ with  a pendant path $P$ at $u$, we say that $G$ is obtained from $H$ by attaching a pendant path $P$ at $u$, where $H=G[V(G)\setminus (V(P)\setminus \{u\})]$. We write $G=H(u,p)$ if the length of $P$ is $p$. Let $H(u,0)=H$.

For a $k$-unifoorm hypergraph $G$ with $u\in V(G)$, and $p\ge q\ge 0$, let $G_u(p,q)=(G_u(p))_u(q)$.

\begin{Lemma}\cite{PL}\label{Lem3} Let $u$ be a vertex of a connected $k$-uniform hypergraph $G$
with $|E(G)|\geq 1$. If $p\geq q\geq 1$, then $\rho(G_u(p,q))>
\rho(G_u( p+1,q-1))$.
\end{Lemma}

Let  $G$  be a connected $k$-uniform hypergraph with $u,v\in V(G)$, and $p\ge q\ge 0$,  let $G_{u,v}(p,q)=(G_u(p))_v(q)$.

\begin{Lemma}  \label{Th1} Let $G$ be a $k$-uniform hypergraph with $k\ge 3$. Let $e$ be a pendant edge of   $G$, and $u$ and $v$ be
 two pendant vertices in $e$. If $p\geq q\geq 1$, then $\rho(G_{u,v}(p,q))>
\rho(G_{u,v}(p+1,q-1))$.
\end{Lemma}

\begin{proof}  Suppose that
 $P=(u_1,e_1,u_2,e_2,\ldots$, $u_{p+1}, e_{p+1},
u_{p+2})$  and $Q=(v_1,f_1,$ $v_2,f_2,\ldots, v_{q-1}, f_{q-1}, v_{q})$ are two
 pendant paths of $G_{u,v}(p+1,q-1)$ at $u$ of length $p+1$ and at $v$
of length $q-1$, respectively, where $u=u_1$ and $v=v_1$. If $q=1$, then let $Q=(v_1)$.

Suppose that $\rho(G_{u,v}(p,q))\le
\rho(G_{u,v}(p+1,q-1))$.
Let $x$ be the principal eigenvector of $G_{u,v}(p+1,q-1)$.

Suppose that  $x_{u_{p+1}}\leq x_{v_{q}}$. Let   $G'$ be the hypergraph obtained from $G_{u,v}(p+1,q-1)$ by
moving edge $e_{p+1}$ from $u_{p+1}$ to $v_{q}$. It is obvious that
$G'\cong G_{u,v}(p,q)$. By Lemma \ref{Lem1}, we have
$\rho(G_{u,v}(p,q))=\rho(G')> \rho(G_{u,v}(p+1,q-1))$, a contradiction. Hence
$x_{u_{p+1}}> x_{v_{q}}$.

Suppose that $i\ge 1$ and $x_{u_{p+1-i}}>x_{v_{q-i}}$ for $i\le q-2$.  We want to show that
$x_{u_{p-i}}>x_{v_{q-i-1}}$.

Suppose that $x_{u_{p-i}}\leq x_{v_{q-i-1}}$. If  $x_{e_{p-i}\setminus\{u_{p-i},u_{p-i+1}\}}>x_{f_{q-i-1}\setminus\{v_{q-i-1},v_{q-i}\}}$.
Let  $U_1=e_{p-i}\setminus \{u_{p-i}\}$ and $V_1=f_{q-i-1}\setminus \{v_{q-i-1}\}$. Let $G'$ be the hypergraph $G'$ obtained from $G_{u,v}(p+1,q-1)$ by deleting edges $e_{p-i}$ and $f_{q-i-1}$ and adding edges $e'_{p-i}$ and $f'_{q-i-1}$, where  $e'_{p-i}=U_1\cup (f_{q-i-1}\setminus V_1)$ and $f'_{q-i-1}=V_1\cup (e_{p-i}\setminus U_1)$.
Note that     $G'\cong G_{u,v}(p,q)$. We have by Lemma
\ref{Lem2} that   $\rho(G_{u,v}(p,q))=\rho(G')> \rho(G_{u,v}(p+1,q-1))$, a contradiction. Thus $x_{e_{p-i}\setminus\{u_{p-i},u_{p-i+1}\}}\leq x_{f_{q-i-1}\setminus\{v_{q-i-1},v_{q-i}\}}$. Now let $U_1=\{u_{p-i+1}\}$ and $V_1=\{v_{p-i}\}$. Let $G'$ be the hypergraph obtained from $G_{u,v}( p+1,q-1)$ by deleting edges $e_{p-i}$ and $f_{q-i-1}$ and adding edges $e'_{p-i}$ and $f'_{q-i-1}$, where $e'_{p-i}=V_1\cup(e_{p-i}\setminus U_1)$ and $f'_{q-i-1}=U_1\cup(f_{q-i-1}\setminus V_1)$. Note that $G'\cong G_{u,v}(p,q)$. By Lemma
\ref{Lem2}, we have  $\rho(G')=\rho(G_{u,v}(p,q))> \rho(G_{u,v}(p+1,q-1))$, also a contradiction. It follows that $x_{u_{p-i}}>x_{v_{q-i-1}}$, as desired.

Let $e=\{u_1,v_1,w_1,w_2,\dots, w_{k-2}\}$ with $d_G(w_1)\ge 2$ and let $e_{p-q+1}=\{u_{p-q+1},u_{p-q+2}, w'_1,$ $w'_2,\dots, w'_{k-2}\}$. Clearly, $x_{w_2}=\cdots=x_{w_{k-2}} $ and  $x_{w'_1}=\cdots=x_{w'_{k-3}}$. If $x_{w_1}\leq x_{w_1'}$, then we can obtain a $G'$  from $G_{u,v}(p+1,q-1)$ by moving all the edges except $e$  incident with $w_1$ from $w_1$ to $w_1'$. By Lemma \ref{Lem1}  and the fact that $G'\cong G_{u,v}(p,q)$, we have $\rho(G_{u,v}(p,q))> \rho(G_{u,v}(p+1,q-1))$, a contradiction. Hence $x_{w_1}>x_{w_1'}$.

Suppose that  $x_{e\setminus\{v_1,w_1\}}\geq x_{e_{p-q+1}\setminus\{u_{p-q+2},w_1'\}}$. Let $U_1=e\setminus\{v_1\}$ and $V_1=e_{p-q+1}\setminus \{u_{p-q+2}\}$. Thus we can form a hypergraph $G'$ from  $G_{u,v}(p+1,q-1)$ by deleting edges $e$ and $e_{p-q+1}$ and adding edges $e'$ and $e'_{p-q+1}$, where $e'=U_1\cup(e_{p-q+1}\setminus V_1)$ and  $e'_{p-q+1}=V_1\cup(e\setminus U_1)$. It is obvious that $G'\cong G_{u,v}(p,q)$. By Lemma \ref{Lem2}, we have $\rho(G_{u,v}(p,q))> \rho(G_{u,v}(p+1,q-1))$, a contradiction. Thus $x_{e\setminus\{v_1,w_1\}}< x_{e_{p-q+1}\setminus\{u_{p-q+2},w_1'\}}$. Now let $U_1=\{w_1\}$ and $V_1=\{w_1'\}$. Let $G'$  be the hypergraph obtained from $G_{u,v}(p+1,q-1)$ by deleting edges $e$ and $e_{p-q+1}$ and adding edges $e'$ and $e'_{p-q+1}$, where $e'=V_1\cup(e\setminus U_1)$ and $e'_{p-q+1}=U_1\cup(e_{p-q+1}\setminus V_1)$. Note that $G'\cong G_{u,v}(p,q)$. By Lemma \ref{Lem2}, we have $\rho(G')=
\rho(G_{u,v}(p,q))> \rho(G_{u,v}(p+1,q-1))$, also  a contradiction. We complete the proof.
\end{proof}

A hypergraph $G$ is said to be reducible if every edge
contains at least one pendant vertex.
For a reducible $k$-uniform hypergraph $G$ with $e\in E(G)$, let $v_e$ be a pendant vertex in $e$, and let $G'$ be the hypergraph with $V(G')=V(G)\setminus \{v_e: e\in E(G)\}$ and $E(G')=\{e\setminus\{v_e\}: e\in E(G)\}$.  We say that $G'$ is reduced from $G$. 

\begin{Lemma}\cite{LM} \label{Lem4} Let $G$ be a reducible $k$-uniform hypergraph. If $G'$ is reduced from $G$, then
$\rho^k(G)=\rho^{k-1}(G')$.
\end{Lemma}

For a $k$-uniform hypertree $G$  with $E(G)=\{e_1, \dots, e_m\}$, if $V(G)=\{v_1, \dots, v_n\}$ with $n=(k-1)m+1$, and $e_i=\{v_{(i-1)(k-1)+1}, \dots, v_{(i-1)(k-1)+k}\}$ for $i=1, \dots, m$, then we call $G$ a $k$-uniform loose path, denoted by $P_m^{(k)}$.

For $k\ge 2$ and $m\ge 3$,
let $D_m^{(k)}$ be the $k$-uniform hypertree obtained from a $k$-uniform loose path $P_{m-1}^{(k)}=\left(u_0,e_1,u_1,\dots ,e_{m-1},u_{m-1}\right)$  by attaching a pendant edge at
$u_1$.

For $k\ge 3$ and $m\ge 3$,
let $D_m'^{(k)}$ be the $k$-uniform hypertree obtained from a $k$-uniform loose path $P_{m-1}^{(k)}=\left(u_0,e_1,u_1,\dots ,e_{m-1},u_{m-1}\right)$  by attaching a pendant edge at
a vertex of degree $1$ in $e_2$.

\begin{Lemma}\cite{LM} \label{Lem5} Let $G$ be a $k$-uniform hypergraph with  $k\geq 3$. 

$(i)$ If $k=3$ and $\rho(G)<\sqrt[k]{4}$, then $G$ is isomorphic to one of the following hypergraphs:
$P_m^{(3)}$ for $m\geq 1$, $D_m^{(3)}$ for $m\geq 3$, $D_m'^{(3)}$
for $m\geq 4$, $B_m^{(3)}$ for $m\geq 5$, $B_m'^{(3)}$ for $m\geq 6$,
$\bar{B}_m^{(3)}$ for $m\geq 7$, $BD_m^{(3)}$ for $m\geq 5$, and
thirty-one  additional hypergraphs: $E_{1,2,2}^{(3)}$,
$E_{1,2,3}^{(3)}$, $E_{1,2,4}^{(3)}$, $F_{2,3,3}^{(3)}$, $F_{2,2,l}^{(3)}$ $($for
$2\leq l\leq 6$$)$, $F_{1,3,l}^{(3)}$ $($for $3\leq l\leq 13$$)$,
$F_{1,4,l}^{(3)}$ $($for $4\leq l\leq 7$$)$, $F_{1,5,5}^{(3)}$, and
$G_{1,1:l:1,3}^{(3)}$ $($for $0\leq l\leq 5$$)$ $($see Figure \ref{fig2}$)$.

$(ii)$ If $k=4$, $G$ is not reducible, and  $\rho(G)\le \sqrt[k]{4}$, then $G\cong H_{1,1,1, t}$ with $t=1,2,3,4$ $($see Figure \ref{fig1}$)$.

$(iii) $ If $k\geq 5$ and $\rho(G)\le \sqrt[k]{4}$, then $G$ is reducible.
\end{Lemma}

\begin{figure}
\centering\footnotesize
\begin{overpic}[scale=0.45]{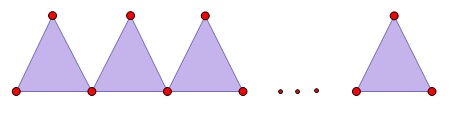}
\put(32,-2){\scriptsize $ P_m^{(3)}\, (m\ge 1)$}
\end{overpic}
\begin{overpic}[scale=0.45]{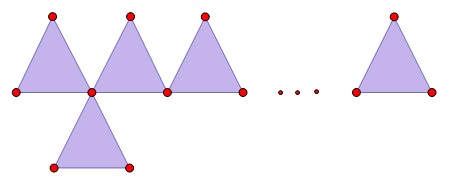}
\put(50,-2){\scriptsize $ D_m^{(3)}\,(m\ge 3)$}
\end{overpic}

\vspace{4mm}

\begin{overpic}[scale=0.45]{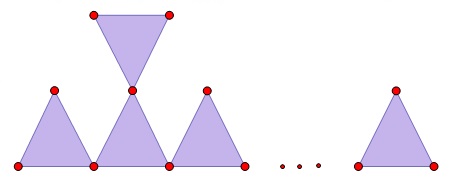}
\put(32,-3){\scriptsize$ D'^{(3)}_m\, (m\ge4)$}
\end{overpic}
\begin{overpic}[scale=0.45]{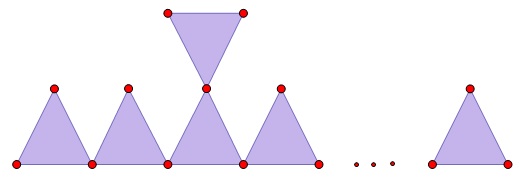}
\put(32,-3){\scriptsize$ B_m^{(3)}\,(m\ge 5)$}
\end{overpic}

\vspace{4mm}

\begin{overpic}[scale=0.4]{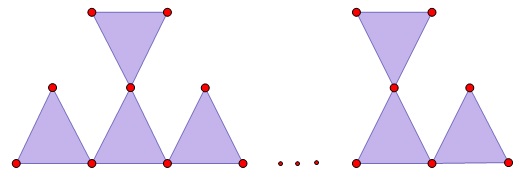}
\put(32,-5){\scriptsize$ B'^{(3)}_m\,(m\ge 6)$}
\end{overpic}
\begin{overpic}[scale=0.4]{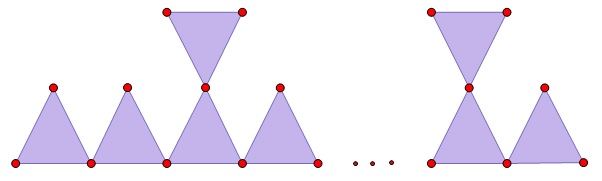}
\put(32,-5){\scriptsize$ \bar{B}_m^{(3)}\, (m\ge7)$}
\end{overpic}

\vspace{4mm}

\begin{overpic}[scale=0.4]{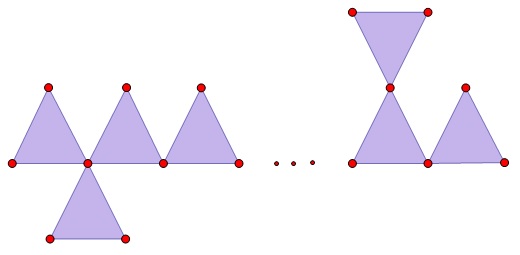}
\put(32,0){\scriptsize$ BD_m^{(3)}\,(m\ge 5)$}
\end{overpic}
\begin{overpic}[scale=0.4]{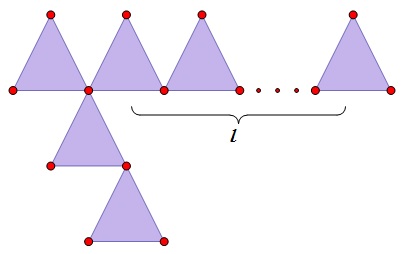}
\put(65,5){\scriptsize$ E_{1,2,l}^{(3)}$}
\end{overpic}

\vspace{4mm}
\begin{overpic}[scale=0.33]{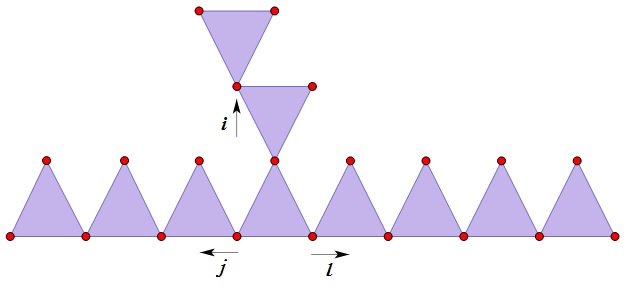}
\put(45,-5){\scriptsize$ F_{i,j,l}^{(3)}$}
\end{overpic}
\begin{overpic}[scale=0.33]{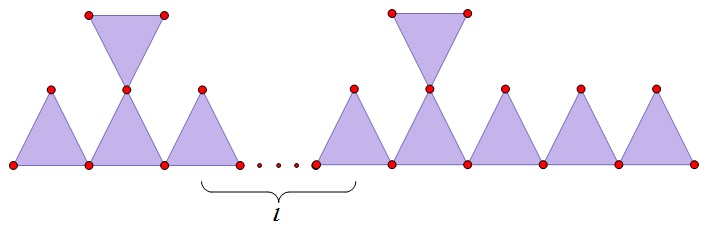}
\put(37,-5){\scriptsize$ G_{1,1:l:1,3}^{(3)}$}
\end{overpic}

\vspace{4mm}

\caption{Hypergraphs in Lemma \ref{Lem5}(i)}\label{fig2}
\end{figure}

\begin{figure}
\centering\footnotesize
\begin{overpic}[scale=0.5]{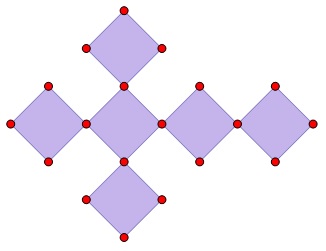}
\put(32,-2){\scriptsize $ H_{1,1,1,2}$}
\end{overpic}
\begin{overpic}[scale=0.5]{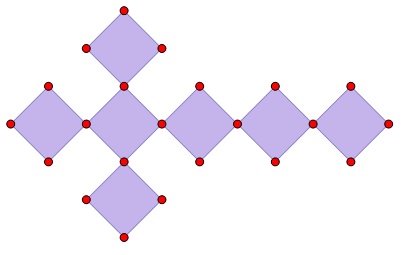}
\put(32,-2){\scriptsize $ H_{1,1,1,3}$}
\end{overpic}

\vspace{5mm}
\begin{overpic}[scale=0.5]{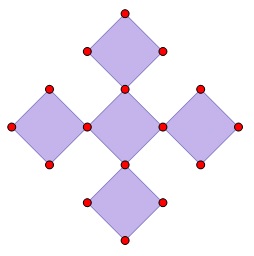}
\put(32,-2){\scriptsize $ H_{1,1,1,1}$}
\end{overpic}
\begin{overpic}[scale=0.5]{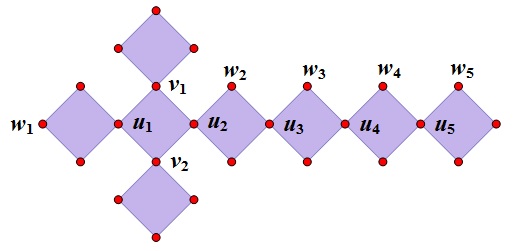}
\put(32,0){\scriptsize $ H_{1,1,1,4}$}
\end{overpic}
\caption{Hypergraphs
$H_{1,1,1,i}$ for $i=1,2,3,4$ in Lemma \ref{Lem5}(ii).}\label{fig1}
\end{figure}

We remark that in Figure~\ref{fig2},  $E^{(3)}_{i,j,l}$ consists of three pendant paths of length $i$, $j$ and $l$ at a common vertex, $F^{(3)}_{i,j,l}$ consists of three pendant paths of length $i$, $j$ and $l$ at three different vertices of a single edge, and $G_{i,j:l:p,q}^{(3)}$ consists of a $3$-uniform loose path of length $l+2$ with two pendant paths of length $i$ and $j$ at two pendant vertices in the first edge and  two pendant paths of length $p$ and $q$ at two pendant vertices in the last edge.

\section{Result}

Now we are ready to show our main result.

\begin{theo}\label{Th2} Let $G$ be a connected $k$-uniform hypergraph with $m\ge 4$ edges, where $k\ge 3$. Suppose that  $G\not\cong P_m^{(k)}$.
Then $\rho(G)\ge \rho(D'^{(k)}_m)$ with equality if and only if
$G\cong D'^{(k)}_m$.
\end{theo}

\begin{proof} By Lemma \ref{Lem5}(i) , $\rho( D_m'^{(3)})<\sqrt[3]{4}$. Then by Lemma \ref{Lem4}, we have $\rho^k(D_m'^{(k)})=\rho^3( D_m'^{(3)})<4$, and thus
$\rho(D_m'^{(k)})<\sqrt[k]{4}$.

Let $G$ be a connected $k$-uniform hypergraph with $m\ge 4$ edges and  $G\not\cong P_m^{(k)}$ having minimum spectral radius. We need only to show that $G\cong D'^{(k)}_m$.

Since $\rho(D_m'^{(k)})<\sqrt[k]{4}$, we have  $\rho(G)<\sqrt[k]{4}$.

\noindent
{\bf Case 1.}  $k=3$.

By Lemma~\ref{Lem5}(i) , $G$ is isomorphic to one of the following hypergraphs:
$P_m^{(3)}$ for $m\geq 1$, $D_m^{(3)}$ for $m\geq 3$, $D_m'^{(3)}$
for $m\geq 4$, $B_m^{(3)}$ for $m\geq 5$, $B_m'^{(3)}$ for $m\geq 6$,
$\bar{B}_m^{(3)}$ for $m\geq 7$, $BD_m^{(3)}$ for $m\geq 5$, and
thirty-one  additional hypergraphs: $E_{1,2,2}^{(3)}$,
$E_{1,2,3}^{(3)}$, $E_{1,2,4}^{(3)}$, $F_{2,3,3}^{(3)}$, $F_{2,2,l}^{(3)}$ $($for
$2\leq l\leq 6$$)$, $F_{1,3,l}^{(3)}$ $($for $3\leq l\leq 13$$)$,
$F_{1,4,l}^{(3)}$ $($for $4\leq l\leq 7$$)$, $F_{1,5,5}^{(3)}$, and
$G_{1,1:l:1,3}^{(3)}$ $($for $0\leq l\leq 5$$)$.

By Lemma \ref{Lem1}, we have  $\rho(D_m^{(3)})>\rho(D'^{(3)}_m)$.

By Lemma \ref{Lem3}, we have $\rho(BD_m^{(3)})>\rho(D'^{(3)}_m)$, and for $E^{(3)}_{1,2,l}$ with $l=2,3,4$, we have $m=l+3$, and thus
$\rho(E^{(3)}_{1,2,l})>\rho(D_{l+3}^{(3)})>\rho(D'^{(3)}_{m})$.

By Lemma~\ref{Th1}, we have
$\rho(B_m^{(3)})>\rho(D'^{(3)}_m)$, $\rho(B'^{(3)}_m)>\rho(D'^{(3)}_m)$, and $\rho(\bar{B}_m^{(3)})>\rho(D'^{(3)}_m)$.

For $F_{i,j,l}^{(3)}$, we have $m=i+j+l+1$. By Lemma~\ref{Th1}, we have
$\rho(F_{i,j,l}^{(3)})>\rho(F_{i,1,m-i-2}^{(3)})>\rho(F_{1,1,m-3}^{(3)})=\rho(D'^{(3)}_{m})$. Thus  $\rho(F_{2,3,3}^{(3)})>\rho(D'^{(3)}_{9})$,
$\rho(F_{2,2,m-5}^{(3)})>\rho(D'^{(3)}_{m})$ for $(7\le m\le 11)$,  $\rho(F_{1,3,m-5}^{(3)})>\rho(D'^{(3)}_{m})$ (for $8\le m\le 18$),  $\rho(F_{1,4,m-6}^{(3)})>\rho(D'^{(3)}_{m})$ (for $10\le m\le 13$), and $\rho(F_{1,5,5}^{(3)})>\rho(D'^{(3)}_{12})$.

By Lemma~\ref{Th1}, we have $\rho(G_{1,1:m-8:1,3}^{(3)})>\rho(F_{1,1,m-3}^{(3)})=\rho(D'^{(3)}_{m})$ for $8\le m\le 13$.

Therefore, if $G\not\cong D'^{(k)}_m$, then $\rho(G)>\rho(D'^{(3)}_m)$. It follows that $G\cong D'^{(3)}_m$.

\noindent {\bf Case 2.} $k=4$.

If $G$ is reducible, then by Lemma~\ref{Lem4}, for the  hypergraph $G_1$ reduced from $G$, we have $\rho(G_1)<\sqrt[3]{4}$, and by the proof in Case~1, we have $G_1\cong D'^{(3)}_m$, implying that $G\cong D'^{(4)}_m$.

Next suppose that  $G$ is not reducible. Then by Lemma \ref{Lem5}(ii),  $G\cong H_{1,1,1, i}$ with $i=1,2,3,4$. We will show that these are impossible.
Since $\rho(D_m^{(3)})>\rho(D'^{(3)}_m)$, by Lemma~\ref{Lem3} we have $\rho(D_m^{(4)})>\rho(D'^{(4)}_m)$. Thus it suffices to show that $\rho(G)>\rho(D^{(4)}_{m})$.

Suppose that $G\cong H_{1,1,1,4}$. Then $m=8$. From the table of \cite{CDS} we have
$\rho(D_{8}^{(2)})=1.962$. By Lemma \ref{Lem3} we have
$\rho(D_{8}^{(4)})=(\rho(D_{8}^{(2)})^{\frac{2}{4}}$, implying that
 $(\rho(D_{8}^{(4)}))^4$ $=3.8494$.

Let $x$ be the principal
eigenvector of $G$, and let $u_i,w_i$ for $i=1,2,3,4,5$ and $v_1,v_2$
be  the vertices of $H_{1,1,1,4}$ as labeled in Figure~\ref{fig1}).
Let $\rho(G)=\rho$ and $x_{u_i}=x_i$ for $i=1,2,3,4,5$. Then
$x_{v_1}=x_{v_2}=x_{u_1}=x_1$. We have $\rho
x_{w_1}^4=x_{w_1}^3x_1$, and then $x_{w_1}=\frac{x_1}{\rho}$.
Similarly, we have  that $x_{w_5}=\frac{x_5}{\rho}$,
$x_{w_i}=\sqrt{\frac{x_ix_{i+1}}{\rho}}$ for $i=2,3,4$. Thus
\begin{equation*}
\begin{aligned}
\rho x_1^4=\frac{x_1^3}{\rho^3}x_1+x_1^3x_2,\\
\rho x_2^4=x_1^3x_2+\frac{x_2^2x_3^2}{\rho},\\
\rho x_3^4=\frac{x_2^2x_3^2}{\rho}+\frac{x_3^2x_4^2}{\rho},\\
\rho x_4^4=\frac{x_3^2x_4^2}{\rho}+\frac{x_4^2x_5^2}{\rho},\\
\rho x_5^4=\frac{x_4^2x_5^2}{\rho}+\frac{x_5^3}{\rho^3}x_5.\\
\end{aligned}
\end{equation*}
From the first two equations, we have $(\rho^2-\frac{\rho^{10}}{(\rho^4-1)^3})x_2^2=x_3^2$, and from the the other three equations, we have
$x_2^2=\rho^2x_3^2-x_4^2=(\rho^6-3\rho^2+\frac{1}{\rho^2})x_5^2$ and $x_3^2=\rho^2x_4^2-x_5^2=(\rho^4-2)x_5^2$. Thus
\[
(\rho^4)^5-8(\rho^4)^4+21(\rho^4)^3-23(\rho^4)^2+13(\rho^4)-3=0.
\]
Since $P_{6}^{(4)}$ is  a subhypergraph of $G$, we have
$\rho^4>\rho^4(P_6^{(4)})=(\sqrt{2\cos\frac{\pi}{8}})^4=2+\sqrt{2}$.

Let $f(t)=t^5-8t^4+21t^3-23t^2+13t-3$.
Note that  $f(\frac{1}{2})=-\frac{3}{2^5}<0, f(2-\sqrt{2})=-7+5\sqrt{2}>0, f(1)=1>0,f(2+\sqrt{2})=-7-5\sqrt{2}<0,f(3.9)=-4.94181<0$,
and $f(4)=1>0$. Thus  $f(t)=0$ has  three real roots $t_1, t_2$ and $t_3$ satisfying  $\frac{1}{2}<t_3<2-\sqrt{2}$, $1<t_2<2+\sqrt{2}$, and $3.9<t_1<4$.
Let $t_4$ and $t_5$ be the remaining two roots of $f(t)=0$. Then $t_4t_5>0$,
$t_4+t_5>8-(2-\sqrt{2})-(2+\sqrt{2})-4=0$,  and $t_4+t_5<8-\frac{1}{2}-1-3.9=2.6<2+\sqrt{2}$. Note that  $\rho^4>2+\sqrt{2}$. So whether $t_4$ and $t_5$ are real or not, they can not be equal to $\rho$.
Thus  $\rho^4=t_1>3.9>3.8494= (\rho(D^{(4)}_8)^4$, i.e., $\rho(G)>\rho(D^{(4)}_8)$, as desired.

In the following, we consider the cases when $G\cong H_{1,1,1,i}$ for $i=1,2,3$.
From the table of \cite{CDS} we have
$\rho(D_{5}^{(2)})=1.902$,  $\rho(D_{6}^{(2)})=1.932$,   and $\rho(D_{7}^{(2)})=1.950$. Then by Lemma \ref{Lem3} we have
$\rho(D_{5}^{(4)})=1.3791$, $(\rho(D_{6}^{(4)}))^4=3.733$, $(\rho(D_{7}^{(4)}))^4=3.8025$.
By similar but simpler  argument as above, we have  $\rho(H_{1,1,1,1})$, $\rho(H_{1,1,1,2}^{(4)})$, and $\rho(H_{1,1,1,3})$ are roots of
$\rho^4-\rho^3-1=0$, $(\rho^4)^4-6(\rho^4)^3+10(\rho^4)^2-7(\rho^4)+2=0$, and
$(\rho^4)^5-7(\rho^4)^4+15(\rho^4)^3-13(\rho^4)^2+6(\rho^4)-1=0$, respectively.
And  we may check that
\begin{equation*}
\begin{aligned}
&\rho(G)> 1.38>1.3791=  \rho(D^{(4)}_5) & \mbox{ if }i=1, \\
&\rho^4(G)> 3.8>3.733=\rho^4(D^{(4)}_6) & \mbox{ if }i=2,\\
&\rho^4(G)> 3.9>3.8025=\rho^4(D^{(4)}_7) & \mbox{ if }i=3,
\end{aligned}
\end{equation*}
i.e., $\rho(G)>\rho(D^{(4)}_m)$, as desired.

 \noindent
 {\bf Case 3.} $k\ge 5$.

By Lemma~\ref{Lem5}(iii) , $G$ is reducible. By Lemma~\ref{Lem4}, for the  hypergraph $G_1$ reduced from $G$, we have $\rho(G_1))<\sqrt[k-1]{4}$. Repeating this process by using Lemmas~\ref{Lem5}(iii)  and \ref{Lem4}, we have a hypergraph sequence $G_0, G_1, \dots, G_{k-4}$ with $G_0=G$, where $G_{i+1}$ is reduced from $G_i$ for $i=0, \dots, k-5$. It is easily seen that $G_{k-4}$ is $4$-uniform and $\rho^k(G)=\rho^4(G_{k-4})<4$. By the proof of Case~2, $G_{k-4}\cong D^{(4)}_m$, and thus $G\cong D^{(k)}_m$.
\end{proof}

Among connected $2$-uniform hypergraphs with $m$ edges, the ones with spectral radius less than $2$ are determined in \cite{Sm} to be the trees $P_m^{(2)}$, $D_m^{(2)}$, and three additional trees with $m=5,6,7$, obtained from  $D_{m-1}^{(2)}$ by attaching a pendant edge at a pendant vertex that is adjacent to a vertex of degree $3$, and by Lemma~\ref{Lem3}, it is easy to see that  $P_m^{(2)}$ for $m\ge 1$ is the unique one with minimum spectral radius, while $D_m^{(2)}$ for $m\ge 3$ is the unique one with second minimum spectral radius.

 Let $G$ be a connected $k$-uniform hypergraph with $2$ edges, where $k\ge 3$. Let $a$ be the number of common vertices of the two edges. Obviously, $1\le a \le k-1$.  By direct calculation, $\rho(G)=2^{\frac{a}{k}}$. 
 Therefore $P_2^{(k)}$ and   the hypergraph in which two edges share two vertices in common  are the unique hypergraphs with minimum and second minimum spectral radii among connected $k$-uniform hypergraphs with exactly $2$ edges.

 Let $G$ be a connected $k$-uniform hypergraphs with  $3$ edges, where $k\ge 3$. If there is a subhypergraph consisting  two edges containing at least two vertices in common, then $\rho(G)\ge \sqrt[k]{4}$. If  any two edges of $G$ contain at most one common vertex, then $G$ is a cycle of length $3$, $D_3^{(k)}$ or $P_3^{(k)}$.   If $G$ is a cycle of length $3$,  then  $\rho(G)=\sqrt[k]{4}$. By Lemma~\ref{Lem3}, $\rho(D_3^{(k)})>\rho(P_3^{(k)})$. By Lemmas~\ref{Lem4} and \ref{Lem5}(i) , $\rho^k(D_3^{(k)})=\rho^3(D_3^{(3)})<4$. Therefore $P_3^{(k)}$ and   $D_3^{(k)}$ are the unique hypergraphs with minimum and second minimum spectral radii among connected $k$-uniform hypergraphs with exactly $3$ edges.

For $m\ge 4$ and $k\ge 3$, by Lemma~\ref{Th1}, we have  $\rho(D'^{(k)}_m)>\rho(P_m^{(k)})$.

Combining the above facts and Theorem~\ref{Th2}, we have

\begin{theo} Among connected $k$-uniform hypergraphs with $m$ edges,
 $P_m^{(k)}$ for $m\ge 1$ is the unique one with minimum spectral radius, and
 the hypergraph in which two edges share two vertices in common for $m=2$ and $k\ge 3$,
$D_3^{(k)}$ for $k=2$ or $m=3$, and $D'^{(k)}_m$ for  $m\ge 4$ and $k\ge 3$
are   the unique ones with  second minimum spectral radius.
\end{theo}

\vskip\baselineskip

\noindent{ \bf {Acknowledgements}}

\vskip\baselineskip  This work was supported by National Natural
Science Foundation of China (No. 11071089 and No. 11701102), Natural Science
Foundation of Guangdong Province (No. 2017A030313032 and
No. 2017A030310441).

\end{document}